\documentclass[leqno,11pt,oneside]{amsart}

\usepackage{geometry}
\usepackage{physics}
\usepackage{times}

\usepackage[all]{xy} 

\usepackage{amsmath, amssymb, amsfonts, latexsym, mdwlist, amsthm,amscd}
\usepackage{caption, subcaption}
\usepackage{tikz}
\usetikzlibrary{calc,trees,arrows,chains,shapes.geometric,%
	decorations.pathreplacing,decorations.pathmorphing,shapes,%
	matrix,shapes.symbols}

\tikzset{
	>=stealth',
	punktchain/.style={
		rectangle,
		rounded corners,
		draw=black, thick,
		minimum height=3em,
		text centered,
		on chain},
	line/.style={draw, thick, <-},
	element/.style={
		tape,
		top color=white,
		bottom color=blue!50!black!60!,
		minimum width=8em,
		draw=blue!40!black!90, very thick,
		text width=10em,
		minimum height=3.5em,
		text centered,
		on chain},
	every join/.style={->, thick,shorten >=1pt},
	decoration={brace},
	tuborg/.style={decorate},
	tubnode/.style={midway, right=2pt},
}

\usepackage{cite}
\usepackage{url}

\usepackage{verbatim}
\usepackage{mathtools}

\usepackage{pgf,tikz}
\usepackage{tikz-cd}
\usetikzlibrary{calc,matrix,patterns,shadings,backgrounds,fit}
\pgfdeclarelayer{myback}
\pgfsetlayers{myback,background,main}
\usepackage{stmaryrd}
\usepackage{diagbox}
\usepackage{extarrows}

\numberwithin{equation}{section} 

\usepackage[bookmarks, colorlinks, breaklinks, pdftitle={}]{hyperref}
\hypersetup{linkcolor=blue,citecolor=blue,filecolor=black,urlcolor=blue}

\usepackage{paralist}
\setdefaultenum{(a)}{(i)}{}{}
\usepackage{enumitem} 

\def\H{\ensuremath{\mathbb{H}}}

\def\Z{\ensuremath{\mathbb{Z}}}

\def\ch{\mathop{\mathrm{ch}}\nolimits}

\def\Coh{\mathop{\mathrm{Coh}}\nolimits}
\def\codim{\mathop{\mathrm{codim}}\nolimits}

\def\dim{\mathop{\mathrm{dim}}\nolimits}

\def\Ext{\mathop{\mathrm{Ext}}\nolimits}

\def\Hom{\mathop{\mathrm{Hom}}\nolimits}
\def\id{\mathop{\mathrm{id}}\nolimits}
\def\Id{\mathop{\mathrm{Id}}\nolimits}
\def\im{\mathop{\mathrm{im}}\nolimits}

\def\min{\mathop{\mathrm{min}}\nolimits}

\def\rk{\mathop{\mathrm{rk}}}

\def\Spec{\mathop{\mathrm{Spec}}}
\def\supp{\mathop{\mathrm{supp}}}

\def\Stab{\mathop{\mathrm{Stab}}\nolimits}

\def\Db{\mathrm{D}^{b}}

\def\Cone{\mathrm{Cone}}

\makeatletter
\newtheorem*{rep@theorem}{\rep@title}
\newcommand{\newreptheorem}[2]{%
\newenvironment{rep#1}[1]{%
 \def\rep@title{#2 \ref{##1}}%
 \begin{rep@theorem}}%
 {\end{rep@theorem}}}
\makeatother

\newtheorem{Thm}{Theorem}[section]
\newreptheorem{Thm}{Theorem}
\newtheorem{Prop}[Thm]{Proposition}
\newtheorem{PropDef}[Thm]{Proposition and Definition}
\newtheorem{Lem}[Thm]{Lemma}

\newtheorem{Cor}[Thm]{Corollary}
\newreptheorem{Cor}{Corollary}

\newtheorem{thm-int}{Theorem}

\theoremstyle{definition}
\newtheorem{Def-s}[Thm]{Definition}
\newtheorem{Def}[Thm]{Definition}
\newtheorem{Rem}[Thm]{Remark}

\def\H{\ensuremath{\mathbb{H}}}

\def\Z{\ensuremath{\mathbb{Z}}}

\def\cC{\ensuremath{\mathcal C}}
\def\cD{\ensuremath{\mathcal D}}
\def\cE{\ensuremath{\mathcal E}}
\def\cF{\ensuremath{\mathcal F}}
\def\cG{\ensuremath{\mathcal G}}

\def\cI{\ensuremath{\mathcal I}}
\def\cJ{\ensuremath{\mathcal J}}

\def\cL{\ensuremath{\mathcal L}}

\def\cO{\ensuremath{\mathcal O}}

\def\cT{\ensuremath{\mathcal T}}

\def\iff{\; \Longleftrightarrow \;}

\usepackage{todonotes}

\def\RHom{\mathrm{RHom}}
\def\sF{\mathsf{F}}
\def\sG{\mathsf{G}}

\begin{document}

\title[Fully faithful functors, skyscraper sheaves, and birational equivalence]{Fully faithful functors, skyscraper sheaves, and birational equivalence}

\author{Chunyi Li}
\address{C. L.:
Mathematics Institute, University of Warwick,
Coventry, CV4 7AL,
United Kingdom}
\email{C.Li.25@warwick.ac.uk}
\urladdr{https://sites.google.com/site/chunyili0401/}

\author{Xun Lin}
\address{X. L.:
Max Planck Institute for Mathematics, Vivatsgasse 7, 53111 Bonn, Germany
}
\email{xlin@mpim-bonn.mpg.de}
\urladdr{}

\author{Xiaolei Zhao}
\address{X. Z.:
Department of Mathematics \\
South Hall 6607 \\
University of California \\
Santa Barbara, CA 93106, USA
}
\email{xlzhao@math.ucsb.edu}
\urladdr{https://sites.google.com/site/xiaoleizhaoswebsite/}


\begin{abstract} 
Let $X$ and $Y$ be two smooth projective varieties such that there is a fully faithful exact functor from $\Db(X)$ to $\Db(Y)$. We show that $X$ and $Y$ are birational equivalent if the functor maps one skyscraper sheaf to a skyscraper sheaf. Further assuming that $X$ and $Y$ are of the same dimension, we show that if $X$ has ample canonical bundle and $H^0(X ,K_X)\neq 0$, or if $X$ is a K3 surface with Picard number one, then $Y$ is birational to a Fourier--Mukai partner of $X$.
\end{abstract}
\date{\today}

\maketitle

\setcounter{tocdepth}{1}
\tableofcontents

\maketitle

\section{Introduction}

The bounded derived category of coherent sheaves $\Db(X)$ of a smooth projective variety $X$  contains a lot of geometric information of the underlying variety. One fascinating aspect of its study is the connection to birational properties. For example, \cite[Conjecture 1.2]{Kawamata:D-equiv_K-equiv} predicts that two smooth projective varieties are derived equivalent if they are K-equivalent. \cite[Conjecture 1.1]{Kuz:fourfold} says that a cubic fourfold is rational if and only if its so-called Kuznetsov component is equivalent to $\Db(S)$ of a K3 surface $S$.

More recently, the question of when derived equivalence implies birational equivalence has been raised in \cite{lieblich2022derived}. In this paper, we provide a result in this direction.

\begin{Thm}[Theorem \ref{thm:pointtoglobal}]
    Let $X$ and $Y$ be two irreducible smooth projective varieties with a fully faithful exact functor $\sF\colon \Db(X)\rightarrow \Db(Y)$. Assume that there exists a closed point $x\in X$ such that $\sF(\cO_x)=\cO_y$ for some closed point $y\in Y$. Then $X$ and $Y$ are birational.
\end{Thm}

The proof of this theorem roughly consists of three steps. First, we show that there exists an open neighborhood $U$ of $x$, such that each skyscraper sheaf in $U$ is sent to a skyscraper sheaf in $Y$. Second, we show that this induces a bijection between closed points in $U$ and closed points in an open subset $V\subset Y$, compatible with the Zariski topology. Finally, we prove that this induces a birational equivalence between $X$ and $Y$.

As an application, we have the following result.

\begin{Thm}[Theorem \ref{thm:K3pic1}]
  Let $S$ be a smooth projective K3 surface of Picard rank one. Then a smooth projective surface $T$ admits $\Db(S)$ as a semiorthogonal factor if and only if $T$ is birational to a Fourier--Mukai partner of $S$.
\end{Thm}

This verifies  \cite[Conjecture 5.8]{Kuz:fractionalCY} ``generically" in dimension $2$. Note that by \cite[Proposition 5.2.5]{abuaf2017compact}, this conjecture is not true in dimension $4$. We need to assume that the K3 surface is of Picard rank one as we are using crucially a result of Bayer and Bridgeland \cite[Proposition 3.15 and 6.3]{K3Pic1}. It would be an interesting question to see how to generalize this to the case of higher Picard rank.

As another application, thanks to the results in \cite{Pirozhkov:delPezzoSOD} and \cite{linxun:remarkhochschild}, we have the following result which is relevant to \cite[Question 2.6]{Kawamata:birg}.

\begin{Cor}[Corollary \ref{cor:pgneq0bir}]
    Let $X$ and $Y$ be two smooth projective varieties with the same dimension. Assume that the canonical divisor $K_X$ is ample and $H^0(X,K_X)\neq 0$. If there exists a fully faithful exact functor $\sF\colon \Db(X)\rightarrow \Db(Y)$, then $Y$ is birational to $X$.
\end{Cor}

\subsection*{Convention}
Throughout the paper, we work over an algebraically closed field $k$ of characteristic $0$. A variety is a reduced separated scheme of finite type over $\Spec k$.

\subsection*{Acknowledgement}
We wish to thank Arend Bayer and Laura Pertusi for their useful comments.

C.L.\  is supported by the Royal Society URF$\backslash$R1$\backslash$201129 “Stability condition and application in algebraic geometry”. X.Z.\ is partially supported by the NSF grant DMS-2101789, and NSF FRG grant DMS-2052665. This paper was written when the third author was attending the Junior Trimester program ``Algebraic geometry: derived categories, Hodge theory, and Chow groups'' at the Hausdorff Institute for Mathematics in Bonn, funded by the Deutsche Forschungsgemeinschaft (DFG, German Research Foundation) under Germany Excellence Strategy – EXC-2047/1 – 390685813. We would like to thank this institution for the warm hospitality.

\section{Support of Semiorthogonal Factor}

\subsection{Adjoint functors}
Let $\cC$ and $\cD$ be a pair of categories, $\sF\colon \cD\to\cC$ and $\sG\colon \cC\to\cD$ be functors. Assume that $\sF$ is left adjoint to $\sG$, namely, for all objects $X$ in $\cC$ and $Y$ in $\cD$, there is a natural bijective 
\begin{equation*}
    \Phi_{YX}\colon \Hom_{\cC}(\sF Y,X)\cong \Hom_{\cD}(Y,\sG X).
\end{equation*}
The map $\Phi$ is natural in the sense that for all morphisms $f\colon X\to X'$ in $\cC$, $g\colon Y'\to Y$ in $\cD$, the following diagram commutes:
\begin{equation}\label{eqadjdiagram}
    \begin{tikzcd}
               \Hom_{\cC}(\sF Y,X)\ar[equal]{rr}{\Phi_{XY}}\ar{d}[swap]{\Hom(\sF g,f)} && \Hom_{\cD}(Y,\sG X)\ar{d}{\Hom(g,\sG f)}\\ 
               \Hom_{\cC}(\sF Y',X')\ar[equal]{rr}{\Phi_{X'Y'}} && \Hom_{\cD}(Y',\sG X')
    \end{tikzcd}
\end{equation}

We will use the following basic facts about adjoint functors.
\begin{Lem}\label{lem:adjfacts}
    Let $X$ and $X'$ be objects in $\cC$, $Y$ and $Y'$ be objects in $\cD$, $f\colon X\to X'$, $h\colon \sF Y\to X$, $h'\colon \sF Y'\to X'$ be morphisms in $\cC$, and $g\colon Y\to Y'$ be morphism in $\cD$. Then
    \begin{enumerate}
        \item [\rm(1)] $\Phi_{YX'}(f\circ h)=\sG(f)\circ \Phi_{YX}(h)$.
        \item[\rm(2)]$\Phi_{YX'}(h'\circ \sF g)= \Phi_{Y'X'}(h')\circ g$.
        \item[\rm(3)] In diagram \ref{eq2comsqs}, the square on the left commutes if and only if the square on the right commutes.
        \begin{equation}\label{eq2comsqs}
    \begin{tikzcd}
               \sF Y\ar{r}{h}\ar{d}[swap]{\sF (g)} & X\ar{d}{f} && Y\ar{rr}{\Phi_{YX}(h)}\ar{d}[swap]{ g} && \sG X\ar{d}{\sG (f)} \\ 
               \sF Y' \ar{r}{h'} & X' && Y' \ar{rr}{\Phi_{Y'X'}(h')}&& \sG X'
    \end{tikzcd}
\end{equation}
    \end{enumerate}
\end{Lem}
\begin{proof}
    \rm{(1)} follows by applying $Y'=Y$ and $g=\id_Y$ in \eqref{eqadjdiagram}. \rm{(2)} is by applying $X=X'$ and $f=\id_{X'}$. \rm{(3)} follows by \rm{(1)} and \rm{(2)} immediately.
\end{proof}

\subsection{Semiorthogonal factor}
Let $\cT$ be a $k$-linear triangulated category. We will be interested in \emph{admissible subcategories} of $\cT$, in other words when the inclusion functor $\alpha\colon\cD\hookrightarrow\cT$  has left adjoint $\alpha^*$ and right adjoint $\alpha^!$. For example, let $X$ and $Y$ be two smooth projective varieties. By \cite[Theorem 1.1]{BondalVdBergh:Generators}, an exact fully faithful functor from $\Db(X)\to\Db(Y)$ realizes $\Db(X)$ as an admissible subcategory of $\Db(Y)$.\\

Assume that $\cD$ is a n admissible subcategory of $\cT$, then $\cT$ admits \emph{semiorthogonal decompositions} as $\langle \cD^\perp,\cD\rangle=\langle \cD,^\perp\!\cD\rangle$. Here $\cD^\perp$ (resp. $^\perp\!\cD$) is the full triangulated subcategory of $\cT$ consisting of objects $\{F\;|\;\RHom(D,F)=0$ for all $D\in\cD\}$ (resp. $\RHom(F,D)=0$).

By  a \emph{semiorthogonal decomposition} (SOD) of $\cT$:
	\begin{equation*}
	\cT = \langle \cD_1, \dots, \cD_m \rangle,
	\end{equation*}
	 we mean a sequence of full triangulated subcategories $\cD_1, \dots, \cD_m$ of $\cT$ such that: 
	\begin{enumerate}
		\item $\RHom(F, G) = 0$, for all $F \in \cD_i$, $G \in \cD_j$ and $i>j$;
		\item For any $F \in \cT$, there is a sequence of morphisms
		\begin{equation*}  
		0 = F_m \to F_{m-1} \to \cdots \to F_1 \to F_0 = F,
		\end{equation*}
		such that $\pi_i(F):=\mathrm{Cone}(F_i \to F_{i-1}) \in \cD_i$ for $1 \leq i \leq m$. 
	\end{enumerate}
The subcategories $\cD_i$ are called the \emph{components} or \emph{SOD factors} of the decomposition.

Assume that $\alpha\colon\cD\hookrightarrow\cT$ is an admissible subcategory of $\cT$. The \emph{left mutation functor} through $\cD$ is the functor $\mathsf{L}_{\cD}:\cT\to \cT$ defined by the canonical distinguished triangle
\begin{equation*}
    \alpha_i\alpha_i^!\xrightarrow{\eta_i}\mathsf{id}\rightarrow \mathsf{L}_{\cD_i},
\end{equation*}
where $\eta_i$ denotes the counit of the adjunction. The image of $\mathsf L_{\cD}$ is  in $\cD^\perp$.

In complete analogy, one defines the \emph{right mutation functor} through $\cD_i$ as the functor $\mathsf{R}_{\cD_i}$ defined by the canonical distinguished triangle
\begin{equation*}
    \mathsf{R}_{\cD_i}\rightarrow\mathsf{id}\xrightarrow{\epsilon_i} \alpha_i\alpha_i^*,
\end{equation*}
where $\epsilon_i$ is the unit of the adjunction.

\subsection{Definition of support}

\begin{Def}\label{def:support}
Let $X$ be a smooth variety and $E\in\Db(X)$. The set-theoretic support $\supp(E)$ is defined to be:
\begin{equation}\label{eq:supp}
   \mathrm{supp}(E):=\{p\text{ closed point on } X\;|\;\RHom(E,\cO_p)\neq 0\}.
\end{equation}
For $i\in\Z$, we write $$\mathrm{supp}^i(E):=\{p\text{ closed point on } X\;|\;\Hom(E,\cO_p[i])\neq 0\}.$$

It is an easy consequence of the ext spectral sequence that 
\begin{equation}\label{eqsuppi}
    \mathrm{supp}^i(E)\subset \bigcup_{j\in \Z} \mathrm{supp}^i(\H^j(E)[-j]).
\end{equation}
\end{Def}

\begin{Lem} \label{lem:suppandgenerator}
Let $\sF\colon \Db(X)\rightarrow \Db(Y)$ be an exact functor. Assume that a subcategory $\cT$ of $\Db(X)$ has a classical generator $G$ in the sense of \cite[Section 3.1]{Rouquier:dimensionoftricat}. Then for every object $E$ in $\cT$, we have $$\supp(E)\subset \supp(G)\text{ and }\supp(\sF(E))\subset \supp(\sF(G)).$$
\end{Lem}

We will use the following facts about the support of objects later.

\begin{Lem}[{\cite[Lemma 5.3]{Bridgeland-Maciocia:K3Fibrations}}]\label{lem:supportbasic}
    Let $E$ be an object and $x$ be a closed point on $X$. Assume that $x\in\supp^0(\H^i(E))\setminus \supp(\tau^{\geq i+1}(E))$, then $\Hom(E,\cO_x[-i])\neq 0$. In particular, $\supp(E)=\cup_{i\in\Z}\supp(\H^i(E))$.
\end{Lem}
\begin{proof}
    Note that $$\Hom(E,\cO_x[-i])\cong\Hom(\tau^{\geq i}(E),\cO_x[-i])\cong\Hom(\tau^{\leq i}(\tau^{\geq i}(E)),\cO_x[-i]),$$
    where the first equality follows from the adjunction $\tau^{\geq i} \dashv \iota^{\geq i}$, and the second uses the assumption that $x\notin \supp(\tau^{\geq i+1}(E))$.
    The statement holds.
\end{proof}

\begin{Lem}\label{lem:jacob}
    Let $R$ be an integral domain with Jacobson radical zero. Let $L_\bullet=\{R^{\oplus a}\to R^{\oplus b}\to R^{\oplus c}\}$ be a sequence of $R$-module such that $L_\bullet\otimes \mathrm{Frac}(R)$ is exact. Then there exists a maximal ideal $\mathfrak m$ such that $L_\bullet \otimes (R/\mathfrak m)$ is exact.
\end{Lem}
\begin{proof}
    Fix a basis for the free modules in $L_\bullet$, and denote the morphisms by matrices $M$ and $N$. When base change to $K:=\mathrm{Frac}(R)$, the exactness of $L_\bullet \otimes K$ is equivalent to $\rk_KM+\rk_KN= b$. 
    
    Let  $M'$ and $N'$ be square sub-matrices of $M$ and $N$ with size $\rk_KM$ and $\rk_KN$ and full rank. The determinants $\det M'$ and $\det N'$ are both non-zero. As the Jacobson radical of $R$ is zero, there exists a maximal ideal $\mathfrak m$ such that $\det M',\det N'\notin \mathfrak m$. Hence, $$b\geq \mathrm{rk}_{R/\mathfrak m}M+\mathrm{rk}_{R/\mathfrak m}N\geq \mathrm{rk}_{R/\mathfrak m}M'+\mathrm{rk}_{R/\mathfrak m}N'=b.$$
    So $L_\bullet\otimes (R/\mathfrak m)$ is exact.
\end{proof}
\begin{Lem}\label{lem:suppi}
    Let $\cF$ be a coherent sheaf on an irreducible smooth variety $X$ with dimension $n$, then 
    \begin{align}\label{eq233}
        \dim(\mathrm{supp}^i(\cF))\leq n-i.
    \end{align}
    In particular, the set $\mathrm{supp}^i(\cF)$ is empty when $i>n$. 
    
    Moreover, let $\Gamma$ be an irreducible component of $\supp(\cF)$ with codimension $m$, then $\supp^m(\cF)\supset \Gamma$.
\end{Lem}
\begin{proof}
    By the upper semicontinuity of $\Hom(\cF,\cO_x[i])$ for closed points $x$ on $X$, the support $\supp^i(\cF)$ is closed. Also by Verdier duality, we have
    \[\H_i(\cF \overset{L}{\otimes} \cO_x) \cong \Hom(\cF,\cO_x[i])^\vee.\]

    As the statement is local, we may assume $X=\Spec{R}$ is affine, and the sheaf $\cF=\widetilde{M}$. As $R$ is regular, the module admits a free resolution $L_\bullet$ of finite length. 

    For every irreducible component $\Spec{R/P}$ of $\supp^i(\widetilde M)$, by definition, the sequence $L_\bullet\otimes R/\mathfrak m$ is not exact at $L_i\otimes R/\mathfrak m$ for every maximal ideal $m\supset P$. 

    Apply Lemma \ref{lem:jacob} to $R/P$ and the sequence $\{L_{i+1}\otimes (R/P)\to L_{i}\otimes (R/P)\to L_{i-1}\otimes (R/P)\}$, we see that the sequence $L_\bullet\otimes \mathrm{Frac}(R/P)=L_\bullet \otimes (R_P/PR_P)$ cannot be exact at the $i$-th term.

    Note that $L_\bullet \otimes R_P$ is a free resolution for the $R_P$-module $M_P$. The cohomology of the sequence $L_\bullet\otimes(R_P/PR_P) $ is $\mathrm{Tor}^\bullet_{R_P}(M_P,R_P/PR_P)$. In particular,  $$\mathrm{Tor}^i_{R_P}(M_P,R_P/PR_P)\neq0.$$
    
    Since $R$ is regular, $R_P$ is regular. It follows that $\mathrm{Tor}^j_{R_P}(M_P,R_P/PR_P)=0$ when $j>\dim(R_P)=n-\dim(\Spec R/P)$. Therefore, $\dim(\Spec R/P)\leq n-i$. The inequality \eqref{eq233} holds.

    For the second part of the statement, as $M$ can be expressed as the extension of the push-forward of some  $R/P$-modules, by the first part of the statement and shrinking the open subset if necessary, we may assume $M\cong R/P$. Then for every maximal ideal $\mathfrak m\supset P$ where $\Spec R/P$ is regular, $\Ext^m_R(R/P,R/\mathfrak m)=k$. The statement follows.
\end{proof}

\begin{Prop}\label{prop:supp2}
    Let $X$ be a smooth variety and $U$ be an open subset of $X$. Let $E\in\Db(X)$ be an object satisfying $\supp(\tau^{\leq 0}(E))\cap U\neq \emptyset$. Then there exists a morphism $h\colon E\to \cO_x[m]$ for some closed point $x\in U$, $m\in\Z$ such that the composition of morphisms $h\circ g\colon \tau^{\leq 0}(E)\to E\to \cO_x[m]$ is non-zero, where $g$ is the canonical morphism associated to the truncation functor.
\end{Prop}
\begin{proof}
    The first step is to choose the closed point $x$. Let
    $$S:=\{\Gamma\;|\; \Gamma \text{ irreducible component of }\mathrm{supp}(\H^i(E))\text{ for some } i\leq 0,\; \Gamma\cap U\neq \emptyset\}.$$
    By assumption, the set $S$ is non-empty. Let $s=\min\{\codim(\Gamma)\;|\; \Gamma\in S\}$ and
    $$t=\min\{i\;|\; \supp(\H^i(E)) \text{ contains an irreducible component }\Gamma \text{ with }\codim(\Gamma)=s, i\leq 0\}.$$

    Let $\Gamma$ be an irreducible component of $\supp(\H^t(E))$ with $\codim \Gamma = s$.  By the choice of $\Gamma$, every irreducible component in $\supp(\tau^{\leq t-1}(E)))\cap U$ has codimension greater than $s$.

    By Lemma \ref{lem:suppi} and \eqref{eqsuppi}, for every $l_2\geq l_1\geq t+1$, we have
    \begin{align*}
        & \dim\left(\mathrm{supp}^{s-t}(\tau^{\geq l_1}(\tau^{\leq l_2}(E))[-1])\right)\leq \dim\left(\bigcup_{l_1\leq j\leq l_2}\mathrm{supp}^{s-t}(\H^j(E)[-1-j])\right)\\
  \notag  =& \max_{l_1\leq j\leq l_2}\{\dim(\mathrm{supp}^{s-t}(\H^j(E)[-1-j]))\}=\max_{l_1\leq j\leq l_2}\{\dim(\mathrm{supp}^{s-t+j+1}(\H^j(E)))\} \\ \notag\leq & \max_{l_1\leq j\leq l_2}\{n-(s-t+j+1)\}\leq n-s-2.
    \end{align*}
Hence there exists a closed point \begin{equation*}
        x\in (U\cap \Gamma)\setminus (\supp(\tau^{\leq t-1}(E))\cup\mathrm{supp}^{s-t}(\tau^{\geq t+1}(\tau^{\leq 0}(E)[-1]))\cup\mathrm{supp}^{s-t}(\tau^{\geq 1}(E)[-1])).
    \end{equation*}

   Now the second step is to show that $\Hom(\tau^{\leq 0}(E),\cO_x[m])\neq 0$, where $m=s-t$. Apply $\Hom(-,\cO_x)$ to $\tau^{\leq t-1}(E)\to\tau^{\leq t}(E)\to \H^{t}(E)[-t]\xrightarrow{+}$, by the choice of $x$, we have $$\Hom(\tau^{\leq t}(E),\cO_x[m])=\Hom(\H^t(E)[-t],\cO_x[m])=\Hom(\H^t(E),\cO_x[s])\neq 0.$$
   
   When $t=0$, this is exactly what we want. When $t\leq -1$, apply $\Hom(-,\cO_x)$ to
   $$\tau^{\geq t+1}(\tau^{\leq 0}(E))[-1]\to\tau^{\leq t}(E)\to\tau^{\leq 0}(E)\to \tau^{\geq t+1}(\tau^{\leq 0}(E))\xrightarrow{+},$$
   by the choice of $x$, we have $\Hom(\tau^{\geq t+1}(\tau^{\leq 0}(E))[-1],\cO_x[m])= 0$, and the non-vanishing follows.

Finally apply $\Hom(-,\cO_x)$ to $\tau^{\geq 1}(E)[-1]\to\tau^{\leq 0}(E)\to E\to \tau^{\geq 1}(E)\xrightarrow{+}$, note that by the choice of $X$, we have $$\Hom(\tau^{\geq 1}(E)[-1],\cO_x[m])= 0.$$ Hence we get a surjective map $$-\circ g\colon \Hom(E,\cO_x[m])\to\Hom(\tau^{\leq 0}(E),\cO_x[m]).$$ 
    
   The preimage of any non-zero morphism in $\Hom(\tau^{\leq 0}(E),\cO_x[m])$ gives $h$ as desired.
\end{proof}

\begin{Lem}\label{lem:2homtosheaf}
    Let $X$ be a smooth variety and $U$ be an open subset of $X$. Let $\cF\in\Coh(X)$ and $p$ be a closed point in $U$.  Let $E\in\Db(X)$ be an object satisfying $\supp(\H^i(E))\cap U=\emptyset$ when $i\geq 0$. Then 
    the composition of any morphisms $E\xrightarrow{g}\cF\xrightarrow{h}\cO_p[m]$ is always zero for every $m\in\Z$.
\end{Lem}
\begin{proof}
     Apply $\Hom(-,\cO_p)$ to the distinguished triangle $\tau^{\leq -1}(E)\xrightarrow{\mathsf{c}} E\xrightarrow{} \tau^{\geq 0}(E)\xrightarrow{+}$, we get
    \begin{equation}\label{eq278}
      \dots\to\Hom(\tau^{\geq 0}(E),\cO_p[m])\to \Hom(E,\cO_p[m])\xrightarrow{-\circ \mathsf c} \Hom(\tau^{\leq -1}(E),\cO_p[m])\to \dots
    \end{equation}
    
    As $\supp(\tau^{\geq 0}(E))=\cup_{i\geq0}\supp(\H^i(E))\not\ni p$, it follows that $\Hom(\tau^{\geq 0}(E),\cO_p[m])=0$.
    
    So the map $-\circ \mathsf c$ in the sequence \eqref{eq278} is injective. Note that $g\circ \mathsf c\in\Hom(\tau^{\leq -1}(E),\cF)=0$, so $h\circ g\circ \mathsf c=0$. It follows that $h\circ g=0$.
\end{proof}

\subsection{Restriction of a functor}

Let $X$ and $Y$ be smooth projective varieties, $U$ and $V$ be an open subvariety on $X$ and $Y$ respectively. Denote by $W=X\setminus U$ and $Z=Y\setminus V$. Denote by $\iota\colon U\to X$ the inclusion morphism.

\begin{Lem}\label{lem:cohresonU}
 Let $\cF_U$ and $\cG_U$ be coherent sheaves on $U$,  $\cF_j$ and $\cG_j$ be coherent sheaves on $X$, $j=1,2$. Assume that $\cF_j|_U\cong \cF_U$ and $\cG_j|_U\cong \cG_U$ for $j=1,2$. Let $f\in \Hom(\cF_1,\cF_2)$, $g\in\Hom(\cG_1,\cG_2)$, $h_j\in\Hom(\cF_j,\cG_j)$ such that $f|_U$, $g|_U$ are isomorphisms and $(h_2\circ f)|_U=(g\circ h_1)|_U$.   
  
  Then there exist coherent sheaves $\cF\subset \iota_*\cF_U$, $\cG\subset \iota_* \cG_U$ and morphisms $f_j\in\Hom(\cF_j,\cF)$, $g_j\in\Hom(\cG_j,\cG)$, such that $f_j|_U$, $g_j|_U$ are isomorphisms and $f_1|_U=(f_2\circ f)|_U$, $g_1|_U=(g_2\circ g)|_U$. In addition, there exists a morphism $h\in\Hom(\cF,\cG)$ such that $h\circ f_1=g_1\circ h_1$ and $h\circ f_2=g_2\circ h_2$.
\end{Lem}
\begin{proof}
    The assumptions give the following commutative diagram:
    \begin{equation}\label{eq3com}
        \begin{tikzcd}
               \iota^*\cF_2 \ar{r}{\widetilde{f}_2}[swap]{\cong} \ar{d}[swap]{\iota^*h_2} & \cF_U \ar{d}{\widetilde{h}}& \iota^*\cF_1\ar{l}{\cong}[swap]{\widetilde f_1} \ar{d}{\iota^*h_1} \\
             \iota^*\cG_2 \ar{r}{\widetilde{g}_2}[swap]{\cong} & \cG_U & \iota^*\cG_1\ar{l}{\cong}[swap]{\widetilde g_1}.
        \end{tikzcd}
    \end{equation}
    As $\iota^*$ is left adjoint to $\iota_*$, by Lemma \ref{lem:adjfacts} (3),  we have the commutative diagram: 
    \begin{equation}\label{eq3com2}
        \begin{tikzcd}
              \cF_2 \ar{r}{\Phi(\widetilde{f}_2)}[swap]{} \ar{d}[swap]{h_2} & \iota_*\cF_U \ar{d}{\iota_*\widetilde{h}} & \cF_1\ar{l}[swap]{ \Phi(\widetilde f_1)} \ar{d}{h_1}  \\
              \cG_2 \ar{r}[swap]{\Phi(\widetilde{g}_2)} & \iota_*\cG_U&\cG_1\ar{l}{\Phi(\widetilde g_1)}.
        \end{tikzcd}
    \end{equation}
    Let $\cF$ be a coherent subsheaf of $\iota_*\cF_U$ containing the images of  $\Phi(\widetilde f_1)$ and $\Phi(\widetilde f_2)$. Let $\cG$ be a coherent subsheaf of $\iota_*\cG_U$ containing the images of $\Phi(\widetilde g_1)$, $\Phi(\widetilde g_1)$, and $(\iota_*\widetilde h)(\cF)$. Let $f_j$, $g_j$, and $h$ be the morphisms restricted from $\Phi(\widetilde f_j)$, $\Phi(\widetilde g_j)$, and $\iota_*\widetilde h$ respectively. It is clear that they satisfy the conditions in the statement.
\end{proof}

\begin{Prop}\label{prop:functorrestoUV}
    Let $\sF\colon \Db(X)\to \Db(Y)$ be an exact  functor satisfying
    \begin{enumerate}
        \item $\supp\left(\H^i(\sF(\cF))\right)\subset Z$ for every $\cF\in\Coh(X)$ and $i\neq0$;
        \item $\supp\left(\sF(E)\right)\subset Z$ for every $E\in\Db(X)$ with $\supp(E)\subset W$.
    \end{enumerate}
   Let $E_j\in\Db(X)$, $j=1,2$ satisfying $\supp\left(\H^i(E_j)\right)\subset W$ for $i\neq 0$. Then we have the following properties: 
   \begin{enumerate}
       \item [\rm(1)] $\H^i(\sF(E))|_V\cong \H^0(\sF(\H^i(E)))|_V$ for every $E\in \Db(X)$ and $i\in\Z$.
       \item[\rm(2)] Consider $f\in\Hom(E_1,E_2)$ such that $\H^0(f)|_U$ is an isomorphism, then $\H^0(\sF(f))|_V$ is an isomorphism.
        If $\H^0(E_1)|_U\cong \H^0(E_2)|_U$, then $\H^0(\sF(E_1))|_V\cong \H^0(\sF(E_2))|_V$.
       \item [\rm(3)] Let $f_j\in\Hom(E_1,E_2)$, $j=1,2$ satisfying $\H^0(f_1)|_U = \H^0(f_2)|_U$, then $\H^0(\sF(f_1))|_V = \H^0(\sF(f_2))|_V$.
   \end{enumerate}
\end{Prop}
\begin{proof}
    \rm{(1)} We induct on the number $n_U(E):=\#\{i\in\Z\;|\;\H^i(E)\cap U\neq \emptyset\}$. 

    When $n_U(E)=0$, we have  $\supp(\H^i(E))\subset W$  for every $i$. Hence $\supp(E)\subset W$. By condition {(\it b)}, 
    $\H^i(\sF(E))|_V=0=\H^0(\sF(\H^i(E)))|_V$.

    When $n_U(E)=1$, we may assume $\H^k(E)\cap U\neq \emptyset$ for some $k\in\Z$. Apply  $\H^\bullet(\sF(-))|_V$ to the distinguished triangle
$\tau^{\leq k}(E)\to E\to\tau^{\geq k+1}(E)\xrightarrow{+}$, we get \begin{align}\label{eq29}
   \dots \to \H^{j-1}(\sF(\tau^{\geq k+1}(E)))|_V \to\H^j(\sF(\tau^{\leq k}(E)))|_V\to   \H^j(\sF(E))|_V\to  \H^j(\sF(\tau^{\geq k+1}(E)))|_V\to  \dots
\end{align}
Note that $\supp(\tau^{\geq k+1}(E))\subset Z$, by (b), $$\H^j(\sF(\tau^{\leq k}(E)))|_V\cong   \H^j(\sF(E))|_V$$
for every $j\in\Z$. Apply $\H^\bullet(\sF(-))|_V$ to the distinguished triangle
$\H^k(E)[-k]\to \tau^{\leq k}(E)\to\tau^{\leq k-1}(E)\xrightarrow{+}$, by (b), we get 
\begin{equation}\label{eq21}
    \H^{j-k}(\sF(\H^k(E)))|_V\cong\H^j(\sF(\H^k(E)[-k]))|_V\cong\H^j(\sF(\tau^{\leq k}(E)))|_V\cong   \H^j(\sF(E))|_V
\end{equation} for every $j\in\Z$. 

When $j\neq k$, by (a), the left-hand side of \eqref{eq21} is $0$, so $\H^j(\sF(E))|_V=0$. Note that $\supp(\H^j(E))\subset Z$, by (b), $\H^0(\sF(\H^j(E)))|_V=0$. When $j=k$, \eqref{eq21} gives the statement directly.

Now assume the statement holds for all $E$ with $n_U(E)\leq s$. Let $E$ be an object with $n_U(E)=s+1$. Assume  that $n_U(\tau^{\leq k}(E)), n_U(\tau^{\geq k+1}(E))\leq s$, then by induction, when $j\leq k$, $$\H^j(\sF(\tau^{\geq k+1}(E)))|_V\cong \H^0(\sF(\H^j(\tau^{\geq k+1}(E))))|_V=0.$$
When $j>k$, $\H^j(\sF(\tau^{\leq k}(E)))|_V\cong \H^0(\sF(\H^j(\tau^{\leq k}(E))))|_V=0.$

By the long exact sequence \eqref{eq29} and induction, when $j\geq k$, 
$$\H^j(\sF(\tau^{\leq k}(E)))|_V\cong \H^0(\sF(\H^j(\tau^{\leq k}(E))))|_V=\H^0(\sF(\H^j(E)))|_V.$$
When $j\leq k+1$, $\H^j(\sF(\tau^{\geq k+1}(E)))|_V\cong \H^0(\sF(\H^j(\tau^{\geq k+1}(E))))|_V=\H^0(\sF(\H^j(E)))|_V.$ 

The statement holds for all objects with $n_U=s+1$ and we complete the induction.\\

\rm{(2)} By the assumptions on $E_1$, $E_2$ and $f$, we have $\supp(\Cone(f))\subset W$. It follows from statement (1) that $\supp(\sF(\Cone(f)))\subset Z$. Apply $\H^0(\sF(-))|_V$ to the distinguished triangle $E_1\xrightarrow{f}E_2\to \Cone(f)\xrightarrow{+}$, the first part of the statement follows.

 For the second part of the statement, by Lemma \ref{lem:cohresonU}, there exists $\cE\in\Coh(X)$ and morphisms $f_j\colon \H^0(E_j)\to\cE$ such that $f_j|_U$ is an isomorphism. Hence $\supp(\Cone(f_j))\subset W$. Apply $\H^0(\sF(-))|_V$ to $\H^0(E_j)\to \cE\to \Cone(f_j)\xrightarrow{+}$, we conclude by (1) and condition (b) that $$\H^0(\sF(E_j))|_V\cong\H^0(\sF(\H^0(E_j)))|_V\cong \H^0(\sF(\cE))|_V.$$

\rm{(3)} By the additivity of all functors, we only need to show that for every $f\in\Hom(E_1,E_2)$,  $\H^0(f)|_U=0$ implies $\H^0(\sF(f))|_V=0$. 

Apply $\H^\bullet(\sF(-))|_V$ to $E_1\xrightarrow{f}E_2\to\Cone (f)\xrightarrow{+}$, by (1), we have the long exact sequence
$$0\rightarrow \H^{-1}(\sF(\Cone(f)))|_V\to \H^0(\sF(E_1))|_V \xrightarrow{\H^0(\sF(f))|_V} \H^0(\sF(E_2))|_V \xrightarrow{g} \H^0(\sF(\Cone(f)))|_V\to 0. $$
As the restricted morphism $f|_U=0$, we have  $\H^0(\Cone(f))|_U\cong \H^0(E_2)|_U$. It follows by (2) that $\H^0(\sF(E_2))|_V \cong \H^0(\sF(\Cone(f)))|_V$. Since $V$ is smooth, in particular Noetherian, and $\H^0(\sF(E_2))|_V$ is coherent, the surjective morphism $g$ must an isomorphism. Therefore, the morphism $\H^0(\sF(f))|_V=0$.
\end{proof}

Let $\cF$ be a coherent sheaf on $U$. Consider the direct system of all pairs of coherent sheaves with morphisms $\{(\cF_i,\varphi_i)\}_{i\in \cI}$ on $X$ indexed by a filtered category $\cI$  such that 
\begin{itemize} \label{eq32}
 \item $\varphi_i\colon \cF_i\rightarrow \iota_*\cF;\;\;\; \varphi_i|_U\colon  \cF_i|_U\cong \iota_*\cF|_U=\cF$.
 \item Two different indices   $i < i'$ if and only if $\varphi_{i'}$ is injective and $\im(\varphi_i)\subset \im(\varphi_{i'})$.
 \item The morphism $\varphi_{ii'}\colon  \cF_i\to \cF_{i'}$ is the unique one satisfyting $\varphi_i=\varphi_{i'}\circ\varphi_{ii'}$.
\end{itemize}

It folows from \cite[\href{https://stacks.math.columbia.edu/tag/01PD}{Lemma 28.22.5 and 6}]{stacksproject} that
\begin{equation}\label{eq_limit}
\;\;  \mathrm{colim}_i(\cF_i)=\iota_*\cF
\end{equation}

\begin{PropDef}\label{propdef:FrestrictoU}
  Let $\sF$ be an exact functor satisfying assumptions as that in Proposition \ref{prop:functorrestoUV}.  Then we may define an additive functor $\sF|_U$  from $\Coh(U)$ to $\Coh(V)$ as follows.

\begin{itemize}[leftmargin=*]
\item 
  Objects: for $\cF\in \Coh(U)$,  define $\sF|_U(\cF)$ to be the coherent sheaf $\H^0(\sF(\cF_i))|_V$ on $V$. 
\item
  Morphisms: for a morphism $h\colon\cF\to \cG$ in $\Coh(U)$, we define a morphism in $\Coh(V)$ as follows. Write $\iota_*\cG=\mathrm{colim}_{j\in \cJ}\cG_j$ as in \eqref{eq_limit}. Then there exist $\varphi_i\colon \cF_i\to \iota_*\cF$, $\phi_j\colon \cG_j \to \iota_*\cG$ and a morphism $h^i_j\colon \cF_i\to\cG_j$ satisfying $\phi_j\circ h^i_j = \iota_*h\circ \varphi_i$. Define $$\sF|_U(h)= \H^0(\sF(h))|_V\colon \H^0(\sF(\cF_i))|_V\to \H^0(\sF(\cG_j))|_V.$$
\end{itemize}

\end{PropDef}
\begin{proof}
   By Proposition \ref{prop:functorrestoUV} (2), the object $\sF|_U(\cF)$ does not depend on the choice of $\cF_i$.  Moreover, for every  $i< k$ in $\cI$ and morphism $\varphi_{ik}\colon \cF_i\to \cF_k$, the identification of $\sF|_U(\cF)$'s is given by the isomorphism $$\H^0(\sF(\varphi_{ik}))|_V\colon \H^0(\sF(\cF_i))|_V\cong\H^0(\sF(\cF_k))|_V.$$
By Proposition \ref{prop:functorrestoUV} (3), the morphism $\H^0(\sF(\varphi_{ik}))|_V$ only depends on $\varphi_{ik}|_U$ hence the embedding $\varphi_i|_U$ and $\varphi_k|_U$. This gives the canonical identifications between $\H^0(\sF(\cF_i))|_V$ for all representatives.\\

    For every $\varphi_i\colon \cF_i\to \iota_*\cF$, the image of  $\iota_*h\circ \varphi_i$ is a coherent subsheaf of $\iota_*\cG$. There exists $\phi_j\colon \cG_j\hookrightarrow \iota_* \cG$ such that $\mathrm{im}(\phi_j)\supset \im(\iota_*h\circ\varphi_i)$. So there exists $h^i_j$ as that in the statement.

    To show that $\sF|_U(h)$ does not depend on the choice of $h^i_j$, let $(\cF_{i'}, \varphi_{i'})$, $(\cG_{j'},\phi_{j'})$ and $h^{i'}_{j'}$ satisfy the assumption $\phi_{j'}\circ h^{i'}_{j'} = \iota_*h\circ \varphi_{i'}$. Then there exists $\varphi_I\colon \cF_I\hookrightarrow \iota_* \cF$ whose $\im(\varphi_I)\supset \im(\varphi_i)\cup\im(\varphi_{i'})$. So there exists $\varphi_{iI}\colon \cF_i\rightarrow\cF_I$ and $\varphi_{i'I}\colon \cF_{i'}\rightarrow\cF_I$.
    
    Let $\phi_J\colon \cG_J\hookrightarrow \iota_* \cG$ whose image contains the image of $\cG_j$, $\cG_{j'}$ and $\iota_*h\circ \varphi_I$. So there exists $h^I_J\colon \cF_I\to \cG_J$ such that     $\phi_J\circ h^I_J=\iota_*\circ \varphi_I$. It follows that $$\left(h^I_J\circ \varphi_{iI}\right)|_U=\left(\phi_{jJ}\circ h^i_j\right)|_U\text{ and }\left(h^I_J\circ \varphi_{i'I}\right)|_U=\left(\phi_{j'J}\circ h^{i'}_{j'}\right)|_U.$$
    By Proposition \ref{prop:functorrestoUV} (3), we have the following commutative diagram:
    \begin{center}
        \begin{tikzcd}
            \H^0(\sF(\cF_i))|_V \ar[equal]{rr}{\H^0(\sF(\varphi_{iI}))|_V} \ar{d}[swap]{\H^0(\sF(h^i_j))|_V}& &\H^0(\sF(\cF_I))|_V\ar{d}{\H^0(\sF(h^I_J))|_V} \ar[equal]{rr}{\H^0(\sF(\varphi_{i'I}))|_V}& &\H^0(\sF(\cF_{i'}))|_V\ar{d}{\H^0(\sF(h^{i'}_{j'}))|_V}\\
            \H^0(\sF(\cG_j))|_V \ar[equal]{rr}[swap]{\H^0(\sF(\phi_{jJ}))|_V} & &\H^0(\sF(\cG_J))|_V\ar[equal]{rr}[swap]{\H^0(\sF(\phi_{j'J}))|_V} & &\H^0(\sF(\cG_{j'}))|_V
        \end{tikzcd}
   \end{center}
Hence, $\sF|_U(h)$ does not depend on the choice of representatives.

The functor $\sF|_U$ is clearly additive as $\H^0(\sF(-))|_V$ is so.
\\

For the identity morphism $\id_{\cF}$, choose representative $\id\colon \cF_i\to\cF_i$, then $\sF(\id)$ is  identity and so does $\sF(\id)|_V$.

Given morphisms $h\colon \cF\to \cG$ and $t\colon \cG\to \cL$ in $\Coh(U)$, we may choose $h^i_j\colon \cF_i\to\cG_j$ and $t^j_k\colon \cG_j\to\cL_k$ as representatives in $\Coh(X)$. It is clear that $t^j_k\circ h^i_j$ is a representative of $t\circ h$. Therefore, $\sF|_U(t\circ h)=\H^0(\sF(t^j_k\circ h^i_j))|_V=\H^0(\sF(t^j_k))|_V\circ\H^0(\sF(h^i_j))|_V=\sF|_U(t)\circ\sF|_U(h).$

In summary, we have checked that $\sF|_U$ is well-defined on objects and morphisms, and $\sF|_U$ satisfies the axioms on morphisms.
\end{proof}
\subsection{Composition of restricted functors}

Let $X_i$, $i=1,2,3$ be smooth projective varieties with open subvarieties $U_i$ respectively. Let $\sF_1\colon \Db(X_1)\to\Db(X_2)$ and $\sF_2\colon \Db(X_2)\to\Db(X_3)$ be exact functors satisfying assumptions as that in Proposition \ref{prop:functorrestoUV} with respect to $U_i$'s. By Proposition and Definition \ref{propdef:FrestrictoU}, we have functors $\sF_1|_{U_1}\colon\Coh(U_1)\to \Coh(U_2)$ and $\sF_2|_{U_2}\colon\Coh(U_2)\to \Coh(U_3)$. We have the following statement for their composition.
\begin{Prop}\label{prop:composition}
Adopt notations as above, then the composition $\sF_2\circ \sF_1\colon\Db(X_1)\to\Db(X_3)$ also satisfies the assumptions in Proposition \ref{prop:functorrestoUV} with respect to $U_1$ and $U_3$. The restricted functor $(\sF_2\circ \sF_1)|_{U_1}=\sF_2|_{U_2}\circ \sF_1|_{U_1}$. 
\end{Prop}
\begin{proof}
    For every coherent sheaf $\cE$ on $X_1$ and $i\neq 0$, by Proposition \ref{prop:functorrestoUV} assumption (b) and statement (2). 
\begin{align*}
      &\supp(\H^i_{X_2}(\sF_1(\cE)))\cap U_2=\emptyset; \\
      &  \H^i_{X_3}((\sF_2\circ\sF_1)(\cE))|_{U_3}\cong \H^0_{X_3}\left(\sF_2\left(\H^i_{X_2}(\sF_1(\cE))\right)\right)|_{U_3}=0.
    \end{align*}
    For every object $E$ in $\Db(X_1)$ with $\supp(E)\cap U_1=\emptyset$, by Proposition \ref{prop:functorrestoUV} assumption (b), $\supp(\sF_1(E))\cap U_2=\emptyset$. By assumption (b) on $\sF_2$ again, $\supp((\sF_2\circ\sF_1)(E))\cap U_3=\emptyset$.
    Therefore, the exact functor $\sF_2\circ\sF_1$ satisfies the assumptions in Proposition \ref{prop:functorrestoUV} with respect to $U_1$ and $U_3$.\\
    
    For every coherent sheaf $\cF$ on $U_1$, let $\cF_{i}$ be a representative of $\iota_{1*}\cF$ on $X_1$, then  $\H^0_{X_2}(\sF_1(\cF_{i}))$ is a representative of $\iota_{2*}(\sF_1|_{U_1}(\cF))$. 
    
    Denote by $\mathsf c_{\geq 0}\colon\sF_1(\cF_i)\to\tau^{\geq0}(\sF_1(\cF_i))$ and $\mathsf c_{ 0}\colon\H^0_{X_2}(\sF_1(\cF_i))\to\tau^{\geq0}(\sF_1(\cF_i))$ the canonical morphisms respectively. We have the isomorphisms:
    \begin{equation}\label{eq3ts}
        \begin{tikzcd}
                 \H^0_{X_3}\left((\sF_2\circ\sF_1)(\cF_i)\right)|_{U_3}\ar{rr}{\H^0_{X_3}(\sF_2(\mathsf c_{\geq0}))|_{U_3}}[swap]{\cong} \ar[equal]{d} && \H^0_{X_3}\left(\sF_2(\tau^{\geq0}(\sF_1(\cF_i)))\right)|_{U_3}  && \H^0_{X_3}\left(\sF_2(\H^0_{X_2}(\sF_1(\cF_i)))\right)|_{U_3} \ar{ll}{\cong}[swap]{\H^0_{X_3}(\sF_2(\mathsf c_{0}))|_{U_3}} \ar[equal]{d}\\
                  (\sF_2\circ\sF_1)|_{U_1}(\cF)\ar[dashed]{rrrr}{\mathsf c_{\cF}}[swap]{=} &&&& \sF_2|_{U_2}(\sF_1|_{U_1}(\cF)).
        \end{tikzcd}
    \end{equation}
    
 As both $\H^0_{X_2}(\mathsf c_{\geq0})|_{U_2}$ and $\H^0_{X_2}(\mathsf c_{0})|_{U_2}$ are isomorphisms,  by Proposition \ref{prop:functorrestoUV}, both $\H^0_{X_3}(\sF_2(\mathsf c_{\geq0}))|_{U_3}$ and $\H^0_{X_3}(\sF_2(\mathsf c_{0}))|_{U_3}$ are isomorphisms.
 
 For any other representative $\cF_j$ with $\varphi_{ik}\colon\cF_i\to\cF_k$, we may apply $\H^0_{X_3}(\sF_2(-))|_{U_3}$ to the following commutative diagrams:
  \begin{equation}\label{eqtaus}
        \begin{tikzcd} 
        \sF_1(\cF_i)\ar{rr}{\mathsf c_{\geq0,i}} \ar{d}{\sF_1(\varphi_{ik})} && \tau^{\geq0}(\sF_1(\cF_i)) \ar{d}{\tau^{\geq 0}(\sF_1(\varphi_{ik}))} && \H^0_{X_2}(\sF_1(\cF_i)) \ar{ll}[swap]{\mathsf c_{0,i}} \ar{d}{\H^0_{X_2}(\sF_1(\varphi_{ik}))}\\
                   \sF_1(\cF_k)\ar{rr}{\mathsf c_{\geq0,j}} &&\tau^{\geq 0}(\sF_1(\cF_k)) && \H^0_{X_2}(\sF_1(\cF_k)) \ar{ll}[swap]{\mathsf c_{0,k}}.
        \end{tikzcd}
    \end{equation}
 
 It follows that the identification $\mathsf{c}_\cF$ between  $(\sF_2\circ\sF_1)|_{U_1}(\cF)$ and $ \sF_2|_{U_2}(\sF_1|_{U_1}(\cF))$ in \eqref{eq3ts} does not depend on the choice of the representatives of $\cF$. \\
 
 For every morphism $h\colon\cF\to\cG$ of coherent sheaves on $U_1$, let $h^i_j\colon\cF_i\to\cG_j$ be a representative of $h$ on $X_1$, then by Proposition and Definition \ref{propdef:FrestrictoU}, $\sF_1|_{U_1}(h)$ has a representative
 $$\H^0_{X_2}(\sF_1(h^i_j))\colon\H^0_{X_2}(\sF_1(\cF_i))\to \H^0_{X_2}(\sF_1(\cG_j))$$
 on $X_2$. 
 
  Denote by $\mathsf c'_{\geq 0}\colon\sF_1(\cG_j)\to\tau^{\geq0}(\sF_1(\cG_j))$ and $\mathsf c_{ 0}\colon\H^0_{X_2}(\sF_1(\cG_j))\to\tau^{\geq0}(\sF_1(\cG_j))$ the canonical morphisms respectively. We have the following  commutative diagram:
    \begin{equation}\label{eqtauhoms}
        \begin{tikzcd} 
        \sF_1(\cF_i)\ar{rr}{\mathsf c_{\geq0}} \ar{d}{\sF_1(h^i_j)} && \tau^{\geq0}(\sF_1(\cF_i)) \ar{d}{\tau^{\geq 0}(\sF_1(h^i_j))} && \H^0_{X_2}(\sF_1(\cF_i)) \ar{ll}[swap]{\mathsf c_{0}} \ar{d}{\H^0_{X_2}(\sF_1(h^i_j))}\\
                   \sF_1(\cG_j)\ar{rr}{\mathsf c'_{\geq0}} &&\tau^{\geq 0}(\sF_1(\cG_j)) && \H^0_{X_2}(\sF_1(\cG_j)) \ar{ll}[swap]{\mathsf c'_{0}}.
        \end{tikzcd}
    \end{equation}
    Apply $\H^0_{X_3}(\sF_2(-))|_{U_3}$ to \eqref{eqtauhoms}, we get 
     \begin{equation}\label{eqtaucom}
        \begin{tikzcd} 
        (\sF_2\circ \sF_1)|_{U_1}(\cF)\ar[equal]{rrrr}{\mathsf c_{\cF}} \ar{d}[swap]{(\sF_2\circ\sF_1)|_{U_1}(h)} &&& &  \sF_2|_{U_2}(\sF_1|_{U_1}(\cF))  \ar{d}{\sF_2|_{U_2}(\sF_1|_{U_1}(h))}\\
                   (\sF_2\circ \sF_1)|_{U_1}(\cG)\ar[equal]{rrrr}{\mathsf c'_{\cG}} &&& &\sF_2|_{U_2}(\sF_1|_{U_1}(\cG)).
        \end{tikzcd}
    \end{equation}
    The statement holds.
\end{proof}

\section{Fully faithful functors and birational equivalence}

\subsection{Main result}

The following theorem is the main result of this paper.

\begin{Thm}\label{thm:pointtoglobal}
    Let $X$ and $Y$ be two irreducible smooth projective varieties with a fully faithful exact functor $\sF\colon \Db(X)\rightarrow \Db(Y)$. Assume that there exists a closed point $x\in X$ such that $\sF(\cO_x)=\cO_y[m]$ for some closed point $y\in Y$ and some integer $m$. Then $X$ and $Y$ are birational.
\end{Thm}

The proof of this result will take up Section \ref{sec:iso_neigh} and \ref{sec:iso_open}. Here we first provide a direct application of this theorem.

\begin{Cor}\label{cor:pgneq0bir}
    Let $X$ and $Y$ be two smooth projective varieties with the same dimension. Assume that the canonical divisor $K_X$ is ample and $H^0(X,K_X)\neq 0$. If there exists a fully faithful exact functor $\sF\colon \Db(X)\rightarrow \Db(Y)$, then $Y$ is birational to $X$.
\end{Cor}
\begin{proof}
    By \cite[Corollary 7.5]{Kuznetsov2009Hochschild},  $\mathsf{HH}_{-\dim X}(X)\cong H^0(X,K_X)\neq 0$ is a direct summand of $\mathsf{HH}_{-n}(Y)\cong H^0(Y,K_Y)$. It follows that $\mathrm{Bs}|K_Y|$ is a proper subset of $Y$. By \cite[Theorem 1.2]{OK:nonsod} and \cite[Lemma 5.3]{Pirozhkov:delPezzoSOD} (or by \cite[Theorem 4.5]{linxun:remarkhochschild} directly), the support of $^\perp(\sF(\Db(X)))$  is contained in $\mathrm{Bs}|K_Y|$. 
    
    Let $y$ be a closed point in $Y\setminus\mathrm{Bs}|K_Y|$, then $\sF(E)=\cO_y$ for some $E\in\Db(X)$. By \cite[Lemma 2.7]{Kuz:fractionalCY}, $$\mathsf{S}_{\sF(\Db(X))}(\sF(E))=(\sF\circ\sF^*)(\mathsf{S}_{\Db(Y)}(\sF(E))=(\sF\circ\sF^*)(\cO_y[n])=\sF(E)[n].$$
    Therefore, $\mathsf{S}_X(E) = E[n]$. Together with the fact that $(E,E)^i=(\cO_y,\cO_y)^i$, by \cite[Proposition 2.2]{Bondal-Orlov}, we must have $E\cong \cO_p[m]$ for some closed point $p$ on $X$. The statement follows by Theorem \ref{thm:pointtoglobal}.
\end{proof}

\begin{Rem}
        By \cite{Olander:CRdim}, if there is a fully faithful exact function from $D^b(X)$ to $D^b(Y)$, then $\dim X\leq \dim Y$.
\end{Rem}

Another application to K3 surfaces of Picard rank $1$ will be given in Section \ref{sec:K3}.

\subsection{Isomorphic of Zariski local neighborhood}\label{sec:iso_neigh}

\begin{Lem}\label{lem:pointrhom}
    Let $X$ be a smooth projective variety, then an object $E\in\Db(X)$ is a skyscraper sheaf if and only if 
    \begin{equation}
        \label{eq:rhomcondition}
   \RHom(\cO_X(-m),E)=k \end{equation} for all $m\gg 0$.
\end{Lem}
\begin{proof}
    The `only if' direction is trivial. For the 'if' direction, let $m$ be sufficiently large so that $\Hom(\cO_X(-m),\H^i(E)[j])=0$ for all $i\in\Z$ and $j\geq 1$. It follows that $$\RHom(\cO_X(-m),E)=\oplus_{i}\Hom(\cO_X(-m),\H^i(E))[-i].$$
    Therefore, the object $E$ is a coherent sheaf.
    
    By Hirzebruch--Riemann--Roch, $\ch_i(E)=0$ for $0\leq i<n$ and $\ch_n(E)=1$ where $n$ is the dimension of $X$. So $E$ must be a skyscraper sheaf.
\end{proof}
\begin{Prop}\label{prop:pointtoopen}
    Let $X$ and $Y$ be two smooth projective varieties. Assume that there exists a fully faithful admissible functor $\sF\colon\Db(X)\rightarrow \Db(Y)$ such that $\sF$ maps the skyscraper sheaf $\cO_x$ on $X$ to $\cO_y$ on $Y$ for a given closed point $x$ on $X$. Then there exist open neighborhoods $U$ of $x$ and $V$ of $y$ such that $\sF$ is a one-to-one correspondence of skyscraper sheaves in $U$ and $V$.
    
    Moreover, the induced map $f\colon U \rightarrow V$ is a homeomorphism between the spaces of closed points with the induced  Zariski topology.
\end{Prop}
\begin{proof}
We may assume that both $X$ and $Y$ are irreducible with the same dimension $n$. 

\noindent\textbf{Step 1:} There is an open subset of $Y$ where all skyscraper sheaves are in the image of $\sF$. (See also \cite[Lemma 2.34]{Pirozhkov:delPezzoSOD}.)

Denote by $\pi\colon\Db(Y)\to \;\!^\perp\!(\mathsf F(\Db(X)))$ the right adjoint functor to inclusion. By \cite[Proposition 7.9]{Rouquier:dimensionoftricat}, the category $\Db(Y)$ has a classical generator $G$. It follows that $\pi(G)$ is a classical generator of $\;\!^\perp\!(\mathsf F(\Db(X)))$. Hence for any object $E$ in $\;\!^\perp\!(\mathsf F(\Db(X)))$, we have $\mathrm{supp}(E)\subset \mathrm{supp}(\pi(G))$.

Note that $\RHom(\pi(G),\cO_y)=0$, we have  $\mathrm{supp}(E)\subset \mathrm{supp}(\pi(G))\not\ni y$ for every $E$ in $\;\!^\perp\!(\mathsf F(\Db(X)))$. Hence for every $z\in Y\setminus \mathrm{supp}(\pi(G))$, the object $\cO_z$ is in $\sF(\Db(X))$.\\

\noindent \textbf{Step 2:} We construct the open subset $V$ in the statement.

Fix an embedding $X\hookrightarrow \mathbf P^N$, choose subspaces $P_0\subset P_1\subset \dots \subset P_n=\mathbf P^N$ such that $P_i\cong \mathbf P^{N-n+i}$ and $X_i:=P_i\cap X$ is smooth with dimension $i$. Moreover, we require $X_i\not\ni x$ for $i<n$. 

For every $i<n$, we have 
$$\RHom(\sF(\cO_{X_i}),\sF(\cO_x))=\RHom(\cO_{X_i},\cO_x)=0.$$
It follows that $\supp(\sF(\cO_{X_i}))\not\ni y$ for every $i<n$. 

For every $m\in\Z$ and object $E$ in $\Db(X)$, we write $E(m):=E\otimes \cO_{\mathbf P^N}(m)$. Note that
\begin{align}\label{eq:Oxextension}
    \sF(\cO_{X_i}(m+1))&\cong\sF(\Cone(\cO_{X_{i-1}}(m+1)[1]\xrightarrow{g}\cO_{X_i}(m))) \\ \notag &\cong \Cone(\sF(\cO_{X_{i-1}}(m+1)[1])\xrightarrow{\sF(g)}\sF(\cO_{X_i}(m)))
\end{align}
for all $0\leq i\leq n$ and $m$. It follows that 
\begin{align*}
&\supp(\sF(\cO_{X_i}(m+1)))\subset \supp(\sF(\cO_{X_i}(m)))\cup \supp(\sF(\cO_{X_{i-1}}(m+1))),\\
&\supp(\sF(\cO_{X_i}(m)))\subset \supp(\sF(\cO_{X_i}(m+1)))\cup \supp(\sF(\cO_{X_{i-1}}(m+1))).
\end{align*}
Note that $\cO_{X_0}(m)=\cO_{X_0}$ for every $m$. By induction, we have
$$\supp(\sF(\cO_{X_{n-1}}(m)))\subset \bigcup_{0\leq i<n}\supp(\sF(\cO_{X_i}))$$
for all $m\in\Z$. Denote by $Z_1=\bigcup_{0\leq i<n}\supp(\sF(\cO_{X_i}))$, then $Z_1$ is a closed subset and it does not contain $y$.

Note that 
$$\RHom(\sF(\cO_X),\sF(\cO_x))=\RHom(\cO_X,\cO_x)=k,$$
the character $\ch_0(\sF(\cO_X))\neq 0$. It follows that $\RHom(\sF(\cO_X),\cO_p)\neq 0$ for any closed point $p\in Y$. By the upper semicontinuity property of the function $\RHom(\sF(\cO_X),-)$ on $X$, there exists an open subset $V_1\subset Y$ such that for every $p\in V_1$,
$$\RHom(\sF(\cO_X),\cO_p)=k.$$
By \eqref{eq:Oxextension} and induction on $m$, for every point $p\in V_1\setminus Z_1$, we have 
\begin{align*}
    \RHom(\sF(\cO_X(m)),\cO_p)=k
\end{align*}
for all $m$. For every $p\in V_1\setminus (Z_1\cup\supp(\sF(G)))$, there exists an object $E_p$ in $\Db(X)$ satisfying $\sF(E_p)=\cO_p$. It follows that
$$\RHom(\cO_X(m),E_p)=\RHom(\sF(\cO_X(m)),\cO_p)=k$$
for every $m\in\Z$. By Lemma \ref{lem:pointrhom}, $E_p=\cO_q$ for some closed point $q$ on $X$.

We let $V=V_1\setminus (Z_1\cup\supp(\sF(G)))$. Denote by $Z=Y\setminus V$. Let
$$U=\{q\in X\;|\;\sF(\cO_q)=\cO_p\text{ for some point $p$ on }V\}.$$

\noindent\textbf{Step 3:} We show that $U$ is an open subset of $X$.

By \cite[Theorem 6.8]{Rouquier:dimensionoftricat}, the category $\Db_Z(Y)$ has  a classical  generator $G_Z$. Denote by $\sF^*$ the left adjoint functor of $\sF$. For any point $q\in U$, 
$$\RHom(\sF^*(G_Z),\cO_q)\cong\RHom(G_Z,\sF(\cO_q))=0.$$
Therefore, for every object $E\in\Db_Z(Y)$, by Lemma \ref{lem:suppandgenerator}, we  have$$\supp(\sF^*(E))\subset\supp(\sF^*(G_Z))\subset X\setminus U.$$

On the other hand, for any point $q'\in X\setminus U$, note that
$$\RHom(\sF(\cO_{q'}),\sF(\cO_q))=\RHom(\cO_{q'},\cO_q)=0$$
for every point $q\in U$. It follows that $\supp(\sF(\cO_{q'}))\cap V =\emptyset$. Choose a point $p'\in \supp(\sF(\cO_{q'})\subset Z$, then 
$$\RHom(\sF^*(\cO_{p'}),\cO_{q'})\cong\RHom(\cO_{p'},\sF(\cO_{q'}))\neq 0.$$
Therefore $p\in \supp(\sF^*(\cO_{p'}))\subset \supp(\sF^*(G_Z))$. It follows that 
$\supp(\sF^*(G_Z))=X\setminus U$ and $U$
 is open in $X$.\\
 
 \noindent\textbf{Step 4:} We show that the induced maps $f$ and $f^{-1}$ are continuous.
 
 Now for every closed point $q$ on $U$, we write $f(q)$ as the unique point on $V$ such that $\sF(\cO_q)\cong \cO_{f(q)}$.
 
For every Zariski closed subset $W$ in $V$, let $\overline{W}$ be its closure in $Y$. Since $\RHom(\sF^*(\cO_{\overline W}),\cO_q)\cong \RHom(\cO_{\overline W},\sF(\cO_q))$, for every $q\in U$, we have 
$$ q\in \supp(\sF^*(\cO_{\overline W}))\cap U\iff f(q)\in \overline W\cap Y = W.$$
It follows that $f^{-1}(W)=\supp(\sF^*(\cO_{\overline W}))\cap U$ and is closed in $U$. Hence, $f$ is continuous with respect to the Zariski topology. A similar argument gives the continuity of $f^{-1}$.

For every Zariski closed subset $M$ in $U$, let $\overline{M}$ be its closure in $X$. Since $\RHom(\cO_{\overline M},\cO_q)\cong \RHom(\sF(\cO_{\overline M}),\cO_{f(q)})$, for every $q\in U$.  It follows that $f(M)=\supp(\sF(\cO_{\overline M}))\cap V$ and is closed in $V$. 

Hence, both $f$ and $f^{-1}$ are continuous with respect to the Zariski topology.
 \end{proof}
 
\subsection{Isomorphism between two open subsets}\label{sec:iso_open}

Adopt the assumptions as that in Proposition \ref{prop:pointtoopen},  denote by $\sF^*$ the left adjoint functor of $\sF$.  Denote by
 \begin{align}
   \label{eq411}  U:=\{p\in X \;| \; \sF(\cO_p)=\cO_q\text{ for some }q\in Y\};\\
  \label{eq412}   V:=\{q\in Y \;| \; \sF(\cO_p)=\cO_q\text{ for some }p\in X\}.
 \end{align}
  By Proposition \ref{prop:pointtoopen}, both $U$ and $V$ are open. Denote by $f\colon U\to V$ the induced map satisfying $\sF(\cO_p)=\cO_{f(p)}$. Denote by $W=X\setminus U$ and $Z=Y\setminus V$. We have the following description for the image of $\Coh(X)$ and $\Coh(Y)$ under the functor $\sF$.

\begin{Prop}\label{prop:Fofcoherentsheaf}
     Adopt notations as above. Let $E$ be a coherent sheave on $X$ , then for $i\neq 0$, $$\supp(\H^i(\sF(E)))\subset Z.$$
     For $i=0$, we have $$\supp(\H^0(\sF(E)))\cap V= f(\supp(E))\cap V .$$
\end{Prop}
\begin{proof}
     For every point $p\in V$, we have
     \begin{equation*}
         \hom(\sF(E),\cO_{p}[i])=\hom(E,\cO_{f^{-1}(p)}[i])=0
     \end{equation*}
     when $i<0$. By Lemma \ref{lem:supportbasic}, 
     \begin{equation}\label{eq37}
         \text{$\supp(\H^i(\sF(E))\subset Z$ when $i>0$. }
     \end{equation}

     For every coherent sheaf $G$ on $Y$ and point $q\in U$, we have 
     \begin{equation*}
         \hom(\sF^*(G),\cO_q[i])=\hom(G,\sF(\cO_q)[i])=\hom(G,\cO_{f(q)}[i])=0
     \end{equation*}
     when $i<0$. By Lemma \ref{lem:supportbasic},  
     \begin{equation}\label{eq38}
         \supp(\H^i(\sF^*(G))\subset W \text{ when $i>0$.} 
     \end{equation}

     Denote by $A$ the truncated object $\tau^{\leq -1}(\sF(E))$ and $g\colon A\to \sF(E)$ the canonical morphism. Denote by $g^*\colon\sF^*(A)\to E$ the adjoint morphism of $g$, then for every point $x\in U$ and  morphism $h\colon E\to \cO_x[m]$. By Lemma \ref{lem:adjfacts} (1), the adjoint morphism of $\sF(h)\circ g\colon A\to \sF(E)\to \sF(\cO_x[m])$ is $h\circ g^*\colon\sF^*(A)\to E\to \cO_x[m]$.

     Since $\H^i(A)=0$ when $i\geq 0$, by \eqref{eq38}, $\supp\left(\H^i(\sF^*(A))\right)\subset W$ when $i\geq 0$. By Lemma \ref{lem:2homtosheaf}, the composition of morphisms $h\circ g^*$ is always zero. 

     Therefore, the morphism $\sF(h)\circ g=0$. As $\sF$ is fully faithful, $\sF(h)$ can be any morphism from $\sF(E)$ to $\cO_x[m]$. By Proposition \ref{prop:supp2}, $\supp(A)\cap U=\emptyset $. So $\supp(\H^i(\sF(E)))=\supp(\H^i(A))\subset Z$ when $i<0$. Together with \eqref{eq37}, we have the statement in the case of $i\neq 0$.

     For the case $i=0$, note that for every point $p\in V$, we have 
     \begin{equation}
    \hom(\H^0(\sF(E)),\cO_{p})=\hom(\sF(E),\cO_{p})=\hom(E,\cO_{f^{-1}(p)}).
     \end{equation}
     The statement holds. 
\end{proof}

\begin{Cor}\label{cor:imageofobject}
    Let $E$ be an object in $\Db(X)$, then 
    $$\supp\left(\H^i(\sF(E))\right)\cap V= f\left(\supp(\H^i(E))\right)\cap V.$$
\end{Cor}

\begin{Prop}\label{prop:Fstarofcoherentsheaf}
     Adopt notations as above. Let  $E$ be a coherent sheave on $Y$, then for $i\neq 0$, $$ \supp(\H^i(\sF^*(E)))\subset W.$$
\end{Prop}
\begin{proof}
  Note that $\Cone\left(E\xrightarrow{\mathrm{can}}\sF(\sF^*(E))\right)\in \;\!^\perp\!(\sF(\Db(X)))$. Since every object in  $^\perp\!(\sF(\Db(X)))$ has support contained in $Z$, we have $\H^i(\sF(\sF^*(E)))\subset Z$ when $i\neq 0$. The statement is followed by Corollary \ref{cor:imageofobject}.
\end{proof}

\begin{Cor}\label{cor:Fstarobjects}
  Let $E$ be an object in $\Db(Y)$, then 
    $$\supp\left(\H^i(\sF^*(Y))\right)\cap U= f^{-1}\left(\supp(\H^i(\sF^*(Y)))\right)\cap U.$$
\end{Cor}

\begin{Prop}\label{prop:UVisom}
    The subvariety $U$ is isomorphic to $V$.
\end{Prop}
\begin{proof}
    By Proposition \ref{prop:Fofcoherentsheaf} and Corollary \ref{cor:imageofobject}, the functor $\sF$ satisfies the assumptions in Proposition \ref{prop:functorrestoUV} with respect to $(U, X)\to(V, Y)$. By Proposition and Definition \ref{propdef:FrestrictoU}, we can define the functor $\sF|_U$ from $\Coh(U)$ to $\Coh(V)$. 
    
    By Proposition \ref{prop:Fstarofcoherentsheaf} and Corollary \ref{cor:Fstarobjects}, the functor $\sF^*$ satisfies the assumptions in Proposition \ref{prop:functorrestoUV} with respect to $(V,Y)\to(U,X)$. By Proposition and Definition \ref{propdef:FrestrictoU}, we can define the functor $\sF^*|_V$ from $\Coh(V)$ to $\Coh(U)$.
    
    By Proposition \ref{prop:composition}, $\sF^*|_V\circ\sF|_U=(\sF^*\circ\sF)|_U=\Id_{\Db(X)}|_U=\Id_{\Coh(U)}$. On the other direction, we have $\sF|_U\circ \sF^*|_V=(\sF\circ\sF^*)|_V$. In the canonical distinguished triangle of functors $\pi\to\Id_{\Db(Y)}\to\sF\circ\sF^*$, as $\supp(\pi(E))\cap V=\emptyset$ for every object in $\Db(Y)$, $(\sF\circ\sF^*)|_V=(\Id_{\Db(Y)})|_V$.

   Therefore, $\Coh(U)$ and $\Coh(V)$ are equivalent. By Gabriel--Rosenberg reconstruction theorem, $U$ is isomorphic to $V$.
\end{proof}

Now Theorem \ref{thm:pointtoglobal} follows directly:

\begin{proof}[Proof of Theorem \ref{thm:pointtoglobal}]
   The statement follows from Proposition \ref{prop:pointtoopen} and \ref{prop:UVisom}. 
\end{proof}

\subsection{K3 surface case}\label{sec:K3}

Thanks to the result in \cite{K3Pic1} on the stability manifold of K3 surface with Picard rank one, our main Theorem \ref{thm:pointtoglobal} also applies to this case.

Let $S$ be a smooth projective K3 surface with Picard rank one. An object $E\in\Db(S)$ is called simple and semirigid if $$\hom(E,E[i])=\begin{cases} 1 & \text{ when } i=0;\\
   2 & \text{ when } i=1;\\
   0 & \text{ when }i<0.\end{cases}$$
    
\begin{Lem}\label{lem:k3pointobj}
    Let $E$ be a simple semirigid object whose character $v_E$ is primitive. Then there exists a Fourier--Mukai partner $S'$ of $S$ and an exact equivalence $\sF\colon \Db(S)\to\Db(S')$ such that $\sF(E)=\cO_p$ for some closed point $p\in S'$.
\end{Lem}
\begin{proof}
   By \cite[Proposition 3.15 and 6.3]{K3Pic1} (see the definition of the width function $w_E$ in the section above Lemma 3.8, the definition of $W_{=0}$ at the beginning of Section 6, and the definition of $W_{\Z}$ above Remark 3.12 in \cite{K3Pic1}), there exists a stability condition $\sigma\in\Stab^\dag(S)$ such that $E$ is $\sigma$-quasistable. As $v_E$ is primitive, $E$ is $\sigma$-stable.
   
   By \cite[Proposition 13.2]{Bridgeland:K3} (more precisely, the fourth paragraph of the proof), there is an autoequivalence $\Phi$ such that skyscraper sheaves are $\Phi(\sigma)$-semistable. By deforming $\sigma$ if necessary, we may assume that skyscraper sheaves are $\Phi(\sigma)$-stable and $\Phi(\sigma)$ is general with respect to $\Phi(v_E)$. In particular, $\Phi(\sigma)$ is in the form of $\sigma_{\beta,\omega}$ as that constructed in \cite[Section 6]{Bridgeland:K3}. The statement then follows from \cite[Theorem 0.0.2]{MYY2} and \cite[Proposition 16.2.3]{Huybrechts:LecturesonK3book}.
\end{proof}

\begin{Thm}\label{thm:K3pic1}
  Let $S$ be a smooth projective K3 surface of Picard rank one. Then a smooth projective surface $T$ admits $\Db(S)$ as a semiorthogonal factor if and only if $T$ is birational to a Fourier--Mukai partner of $S$.
\end{Thm}
\begin{proof}
    The `if' direction follows from Orlov's blowup formula \cite[Theorem 4.3]{Orlovblowup}.
    
   For the `only if' direction, denote by $\sF\colon \Db(S)\to\Db(T)$ the fully faithful exact functor. By the same argument as that in Corollary \ref{cor:pgneq0bir}, for every closed point $y\in T\setminus \mathrm{Bs}|K_T|$, the skyscraper sheaf $\cO_y$ is $\sF(E)$ for some  $E\in\Db(S)$ with $\hom(E,E[i])=\hom(\cO_y,\cO_y[i])$. As the character of $\cO_y$ in $H^*_{\mathrm{alg}}(T,\Z)$ is primitive, the character of $E$ is primitive. By Lemma \ref{lem:k3pointobj}, we may assume that there is a Fourier--Mukai partner $S'$ of $S$ and fully faithful exact functor $\sF'\colon \Db(S')\to \Db(T)$ so that $\sF'(\cO_p)=\cO_y$ for some closed point $p$. The statement follows from Theorem \ref{thm:pointtoglobal}.
\end{proof}

\section{An Alternative approach to Corollary \ref{cor:pgneq0bir}}

In this section, we provide an alternative approach to Corollary \ref{cor:pgneq0bir} without using Theorem \ref{thm:pointtoglobal}. In this case, with a simpler proof we can deduce a more precise statement on which open sets are isomorphic.

\begin{Thm}\label{thm:mainYminusBsKtoX}
 Let $X$ and $Y$ be smooth projective varieties of the same dimension $n$. Assume that $K_{X}$ is ample, $H^0(X,K_X)\neq 0$ and there is a fully faithful functor $\sF:D^{b}(X)\hookrightarrow D^{b}(Y)$, then there is an embedding $$f\colon Y\setminus \operatorname{Bs}\vert K_{Y}\vert \hookrightarrow X.$$
 In particular, $X$ is birational to $Y$.
\end{Thm}
\begin{proof}
For every closed point $y\in U=Y\setminus \operatorname{Bs}\vert K_{Y}\vert$, by the same argument as that in Corollary \ref{cor:pgneq0bir}, there exists a point like object $E_{y}\in D^{b}(X)$ such that $\sF(E_{y})=\cO_y$ and $E_{y}\cong \cO_x[m_{y}]$ for some closed point $x$ and some integer $m_{y}$. This induces a set-theoretic map 
$$f\colon U\rightarrow X\;\;\;\colon\;\;\; y\mapsto x.$$
We show that the integer $m_y$ is independent of closed point $y$.

By \cite[Theorem 1.1.]{BondalVdBergh:Generators} and \cite[Theorem 2.2]{Orlov:representability}, the functor $\sF$ and its adjoint are of Fourier--Mukai type. Denote their kernel as $P$ and $P_R$ respectively.

We write $j_y: X\rightarrow X\times Y,\quad x\mapsto (y,x)$. There is a commutative diagram,
$$\xymatrix{y\times X\ar[d]\ar[r]^{j_y}& Y\times X\ar[d]\\
y\ar[r]& Y}$$
Identify $y\times X$ with $X$, we have
$$\cO_{f(y)}[m_y]\cong E_{y}\cong \mathbb{L}j_y^{\ast}P_{R}.$$
Recall that $\mathbb{L}j_y^{\ast}P\cong j_y^{-1}P\otimes^{\mathbb{L}}_{j_y^{-1}\mathcal{O}_{X\times Y}}\mathcal{O}_{X}$, there is a spectral sequence
$$E^{r,s}_{2}=\mathrm{Tor}_{-r}(\H^{s}(j_y^{-1}P),\mathcal{O}_{X})\Rightarrow \mathrm{Tor}_{-(r+s)}(j_y^{-1}P,\mathcal{O}_{X}).$$
We write $Z_{q}=\supp(\H^{q}(P_{R}))$, then $\supp(P_{R})=\bigcup_{q\in\Z} Z_{q}$ and
$$\supp(P_{R})\vert_{U\times X}=\{(y,f(y))\vert y\in U\}.$$
Consider the projection map
$$p_{U}:U\times X\rightarrow U.$$
Let $Z'_{q}=p_{U}(Z_{q}\cap U\times Y)$ which are closed subsets. Then $U=\bigcup_{q\in\Z} Z'_{q}$.
Since $U$ is irreducible, we have $Z'_{q}=U$ or $Z'_{q}=\emptyset$. Observe that $j^{-1}(\H^{q}(P_{R}))$ and $\mathcal{O}_{X}$ are sheaves, hence $E^{r,s}_{2}=0$ for $r\leq -1$. Choose the minimal $q$ with $Z'_{q}=U$, then the second page $E^{0,q}_{2}$ survives to $E^{0.q}_{\infty}$. Note that $q$ is independent of the point $y$. Thus we must have
$$\sF^*(\cO_y)\cong \cO_{f(y)}[q].$$
Finally, replacing $\sF$ with $\sF[q]$, then the right adjoint $\sF^*[-q]$ maps $\cO_y$ to $\cO_{f(y)}$. So we may assume $q=0$. Consider the diagram 
$$\xymatrix{&\supp(P_{R}\vert_{U\times X})\ar[dl]_{p_{U}}\ar[dr]^{p_{X}}&\\
U&&X}$$
Since $p_{U}$ is proper and quasi-finite, hence it is finite of degree one. Therefore $p_{U}$ is an isomorphism. Thus, $f=p_{X}\circ p^{-1}_{U}: U\rightarrow X$ is a morphism of algebraic varieties. 
\begin{Prop}\label{injective}
 $f$ is injective.
\end{Prop}
\begin{proof}
   Take two closed point $y_{1},y_{2}\in U$. Suppose $f(y_{1})=f(y_{2})$, then $\sF^*(\cO_{y_{1}})=\sF^*(\cO_{y_{2}})$. Apply the functor $\sF$ to both sides, we have
   $\cO_{y_{1}}\cong \cO_{y_{2}}$ since $\cO_{y_{1}}$ and $\cO_{y_{2}}$ are in the image of functor $\sF$. Thus $y_{1}=y_{2}$.
\end{proof}
We then prove that $f$ is an open embedding by dividing the proof into lemmas, which implies $X$ is birational to $Y$. 
\begin{Lem}\label{proj}
 The projection $\supp(P)\rightarrow X$ is surjective.  
\end{Lem}
\begin{proof}
    For every closed point $x$, we write $j_{x}': Y\hookrightarrow X\times Y$ as the embedding $y\mapsto (x,y)$.
    Suppose the projection is not surjective, then there exists closed point $x$ such that $j'_{x}(Y)\cap\supp(P)=\emptyset$, therefore $\mathbb{L}j'^{\ast}_{x}P=0$. Thus $\sF(\cO_x)\cong0$, contradict that $\sF$ is fully faithful.
\end{proof}
\begin{Lem}\label{pointobject}\cite[Lemma 4.5]{FourierMuaitransformsinaglebraicgeometry06}
   Let $Z$ be a smooth projective variety. Let $E$ be any object in $D^{b}(Z)$ such that 
   $$Hom(E,E[t])=\left\{
   \begin{aligned}
       & k, t=0\\
       & 0, t\leq -1
   \end{aligned}\right.$$
   Suppose the support of $E$ is zero-dimensional, then $E$ is of the form $\cO_z[m]$ for some $z$ and integer $m$.
\end{Lem}
\begin{Lem}\label{supportkernel}
  $\supp(P)=\supp(P_{R})$. 
\end{Lem}
\begin{proof}
   Since $P_{R}\cong P^{\vee}\otimes p^{\ast}_{X}(K_{X})$, see for example \cite[Proposition 5.9]{FourierMuaitransformsinaglebraicgeometry06}, hence $\supp(P_{R})=\supp(P^{\vee})=\supp(P)$. 
\end{proof}
Let $Z_{1}=\overline{\supp(P_{R}\vert_{U\times X})}$, according to Lemma \ref{supportkernel}, we have $\supp(P)=Z_{1}\cup(\bigcup_{j\neq 1} Z_{j})$, where $Z_{j}$ are other irreducible components. According to Lemma \ref{proj}, the projection $\supp(P)\rightarrow X$ is surjective. 
\begin{Lem}\label{Z1sujective}
  The projection $p_{X}:Z_{1}\rightarrow X$ is surjective.   
\end{Lem}
\begin{proof}
    First, the image of $Z_{j}\vert_{j\neq 1}$ is disjoint with image of $f$. Otherwise, let $x'\in \operatorname{Image}(f)\cap  {Z_{j}}$, and $x'=f(y')$. Since $\sF^*(\cO_{y'})\cong \cO_{x'}$ and $\cO_{y'}$ is in the image of functor $\sF$, we have $\sF(\cO_{x'})\cong \cO_{y'}$. Therefore, the fiber of projection $\supp(P)\rightarrow X$ over $x'$ is $y'$. Hence $(y',x')\in Z_{j}\cap U\times X\neq \emptyset$. Consider the projection to $U$, by Lemma \ref{supportkernel} again, the fiber of every point $y\in U$ is $(y,f(y))$, this implies $Z_{j}\vert_{U\times X}\subset Z_{1}\vert_{U\times X}$, and hence $Z_{j}\subset Z_{1}$, a contradiction.
    
    It follows that we have $X=f(Z_{1})\cup(\bigcup_{j\neq 1} f(Z_{j}))$ by Lemma \ref{proj}. Since $X$ is irreducible, and all $f(Z_{j})$ are proper closed subsets, $X=f(Z_{1})$.
\end{proof}

Take $x_{0}\in \operatorname{Im}(f)$. Let $x_{0}=f(y_{0})$, then $\sF^*(\cO_{y_0})\cong \cO_{x_{0}}$. Hence $\sF(\cO_{x_{0}})\cong \cO_{y_{0}}$. Therefore, the fiber of the projection $Z_{1}\rightarrow X$ over $x_{0}$ is $y_{0}$. Then by Lemma \ref{Z1sujective}, there exists an open neighborhood $V_{x_{o}}\subset X$ such that the fiber over the closed point
in $V_{x_{0}}$ is zero-dimensional. We can choose $V_{x_{0}}$ that is disjoint with the image of $Z_{j}\vert_{j\neq 1}$. The support of $\sF(\cO_{x})$ is zero-dimensional for $x\in V_{x_{0}}$. According to Lemma \ref{pointobject}, $\sF(\cO_x$ is of the form $\cO_{y_{x}}[m_{x}]$. Again these $m_{x}$'s are independent of  $x$ locally, hence $m_{x}=m_{x_{0}}=0$. Write $Z_{1}=Z_{1}\vert_{U\times X}\cup Z'_{1}$, where $Z'_{1}$ is a closed subset of $Z_{1}$. Since the image of $Z'_{1}$ under projection $Z_{1}\rightarrow X$ does not contain $x_{0}$, hence by shrinking the open neighborhood $V_{x_{0}}$, we can assume $f(Z'_{1})\cap V_{x_{0}}=\emptyset$. Thus the fiber of $V_{x_{0}}$ is in $U\times X$.

Due to the same reason, there is a morphism $h_{0}:V_{x_{0}} \rightarrow U$. Since $\sF^*(\cO_{h(x)})\cong \sF^*\circ \sF(\cO_{x})\cong \cO_x$ for $x\in V_{x_{0}}$, therefore
$V_{x_{0}}\subset \operatorname{Image}(f)$. In particular, $\operatorname{Image}(f)$ is an open subset of $X$. Finally by gluing the morphism $h_{0}$, we obtain morphism 
$$h: \operatorname{Image}(f)\rightarrow U.$$
By construction, the composition $f\circ h=\id$. By Proposition \ref{injective}, $f$ is injective. Hence  $f$ is an isomorphism into the image with inverse $h$.
\end{proof}                                                                                                     
\bibliography{all}                      
\bibliographystyle{halpha}  
\end{document}